\documentclass[11pt]{amsart}
\usepackage{amsmath, amsthm, amssymb, latexsym,url}
\usepackage{hyperref}

\DeclareFontFamily{U}{mathx}{\hyphenchar\font45}
\DeclareFontShape{U}{mathx}{m}{n}{<-> mathx10}{}
\DeclareSymbolFont{mathx}{U}{mathx}{m}{n}
\DeclareMathAccent{\widebar}{0}{mathx}{"73}

\author{P. Pollack and J. Vandehey}
\thanks{Email: \href{mailto:vandehey@uga.edu}{\nolinkurl{vandehey@uga.edu}}}
\title{Besicovitch, bisection, and the normality of $0.(1)(4)(9)(16)(25)\dots$}
\date{\today}

\newtheorem{thm}{Theorem}

\newtheorem*{claim}{Claim}

\newtheorem{prop}[thm]{Proposition}
\newtheorem{lem}[thm]{Lemma}
\theoremstyle{definition}

\theoremstyle{remark}

\DeclareMathAlphabet{\curly}{U}{rsfs}{m}{n}

\newcommand{\s}{\mathbf{s}}

\begin{document}

\maketitle

\allowdisplaybreaks

\begin{abstract}
We revisit Besicovitch's 1935 paper in which he introduced several techniques that have become essential elements of modern combinatorial methods of normality proofs. Despite his paper's influence, the results he inspired are not strong enough to reprove his original result. We provide a new proof of the normality of the constant $0.(1)(4)(9)(16)(25)\dots$ formed by concatenating the squares, updating Besicovitch's methods.
\end{abstract}

%%%%%%%%%%%%
\section{Introduction}\label{section:introduction}
%%%%%%%%%%%%%%%%%

A real number $x$ is said to be normal (to base $10$) if every string of decimal digits appears in the decimal expansion of $x$ as frequently as every other string of the same length, so the digit $4$ should appear as often as $9$, and $299$ should appear as often as $058$. More concretely, given a fixed integer base $g \ge 2$, let $\nu(x,N,\s)$ denote the number of times the string $\s$, consisting of $k$ base $g$ digits, appears in the first $N$ digits of the base $g$ expansion for $x$: then $x$ is normal to base $g$ if for every non-empty string $\s$, we have
\begin{equation}\label{eq:mainlimit}
\lim_{N\to \infty} \frac{\nu(x,N,\s)}{N} = \frac{1}{g^k}.
\end{equation}

Borel showed that almost all real numbers are normal to a given base $g$, in the sense that the set of numbers that are not normal has Lebesgue measure $0$. Despite this, to this day, no well-known mathematical constant, such as $e$, $\pi$, or $\ln 2$, is known to be normal to any integer base. All known examples of normal numbers were numbers constructed to be normal.

The first such explicit construction was given by Champernowne \cite{champernowne33}: he showed that if we concatenate all the integers in succession---like so, $ 0.123456789101112\cdots$---then the resulting number is normal to base $10$. Champernowne's result inspired many mathematicians to look at sequences of positive integers $\{a_n\}_{n=1}^\infty$ which make the number $0.\widebar{a_1}\widebar{a_2}\widebar{a_3}\widebar{a_4}\cdots$ normal in a given base. Here, given an integer $a$, we shall let $\widebar{a}$ denote the string composed of its base $g$ digits.

Shortly after Champernowne, Besicovitch studied the sequence $a_n=n^2$. Although it is commonly stated that Besicovitch proved that the number $x_B=0.14916253649\cdots$ is normal to base $10$, in fact he showed a different result from which the normality of $x_B$ can be derived relatively quickly. Nonetheless, Besicovitch's work was important for two main reasons.

First, Besicovitch's result inspired Davenport and Erd\H{o}s \cite{DE52} to develop a method of proving normality through exponential sum estimates, and they used this method to show that if $f(n)$ is a positive, non-constant, integer-valued polynomial, then $0.\widebar{f(1)}\widebar{f(2)}\widebar{f(3)}\dots$ is a normal number to the given base. (Taking $f(n)=n^2$, this gives a second proof that $x_B$ is normal to base $10$.) This has given rise to what we call the \emph{analytic method of normality proofs}, which has resulted in a number of different normality results by applying different results on exponential sums (see, for example, \cite{MTT08,NS97,vandehey13}).

Second, Besicovitch's result inspired the definition of an integer being $(\epsilon,k)$-normal. Let $\nu(a,\s)$ denote the number of times the finite string $\s$ appears in $\widebar{a}$, and let $L(a)$ denote the number of digits in the string $\widebar{a}$. Then an integer $n$ is said to be $(\epsilon,k)$-normal if 
\[
\left| \frac{\nu(a,\s)}{L(a)} - \frac{1}{g^k}\right| \le \epsilon
\]for every string $\s$ with $k$ digits. This definition inspired the following result, which may be called \emph{the combinatorial method of normality proofs}.

\begin{thm}\label{thm:combinatorial}Consider a sequence $\{a_n\}_{n=1}^\infty$. 
Suppose that the lengths of the strings $\widebar{a_n}$ are growing on average, but that no one length dominates; more precisely, suppose that as $m$ tends to infinity, we have 
\begin{equation}\label{eq:lengths}
m = o\left( \sum_{n=1}^m L(a_n) \right) \qquad \text{ and } \qquad m\cdot \max_{1\le n \le m} L(a_n) =O\left( \sum_{n=1}^m L(a_n)\right) .
\end{equation} In addition, suppose that for any fixed $\epsilon>0$ and $k\in \mathbb{N}$, almost all $a_n$ are $(\epsilon,k)$-normal, in the sense that the number of $n\le m$ for which $a_n$ is \emph{not} $(\epsilon,k)$-normal is $o(m)$ as $m$ tends to infinity again.

Then $x=0.\widebar{a}_1\widebar{a}_2\widebar{a}_3\widebar{a}_4\dots$ is normal.
\end{thm}

Here, we use the notation $f(x)=O(g(x))$ to mean that $|f(x)/g(x)|\le C $ for some constant $C$ (called an \emph{implicit constant}), and the notation $f(x)=o(g(x))$ to mean that $f(x)/g(x)$ approaches $0$ as $x$ approaches infinity. We shall also use the notation $f(x) \sim g(x)$ to mean $f(x) = g(x)(1+o(1))$ or, equivalently, $\lim_{x\to \infty} f(x)/g(x)=1$. 

Theorem \ref{thm:combinatorial} says, in essence, that if almost all $a_n$ exhibit small-scale normality results, then we expect the full number $x$ to exhibit large-scale normality results. Although Theorem \ref{thm:combinatorial} is implicitly used in almost every combinatorial normality proof, we are unaware of it ever being given explicitly in the literature, and so provide a proof in Section \ref{sec:combproof}.

Copeland and Erd\H{o}s \cite{CE46} gave a fairly powerful counting result on the number of integers that are \emph{not} $(\epsilon,k)$-normal. The first half of the following proposition is due to them; the second half is derived from Lemma $4.7$ in \cite{Bugeaud}.

\begin{prop}\label{prop:CE}
Let $\epsilon>0$ and $k\in \mathbb{N}$ be fixed. 

There exists a $\delta=\delta(\epsilon,k)>0$ such that the number of integers in the interval $[1,m]$ that are not $(\epsilon,k)$-normal is at most $m^{1-\delta}$ for all sufficiently large $m$.

There also exists a $\delta'=\delta'(\epsilon,k)>0$ such that the number of base-$g$ strings of length $\ell$ (including those that start with $0$) that are not $(\epsilon,k)$-normal is at most $g^{\ell(1-\delta)}$ for all sufficiently large $\ell$.
\end{prop}

Combining Theorem \ref{thm:combinatorial}, the first half of Proposition \ref{prop:CE}, and the Prime Number Theorem, it is immediate that the Copeland-Erd\H{o}s number $0.2357111317\dots$, formed by taking $a_n$ to be the $n$th prime, is normal to base $10$.

Many early normality results focused on sequences $\{a_n\}_{n=1}^\infty$ that were increasing. More recent variants have allowed more chaotic and oscillating functions $a_n=f(n)$ to be considered. Recently, the authors of this note looked at functions $f(n)$ that are \emph{almost bijective}, which we shall define in the following way.

First, we say a set $S\subset \mathbb{N}$ is \emph{meager} if $\#\{n\in S: n \le m \} \le m^{1-\delta}$ for some fixed $\delta>0$ and all sufficiently large $m$. We say a set $S\subset \mathbb{N}$ has \emph{asymptotic density $0$} if $\#\{n\in S: n \le m \} =o(m)$. We say a function $f\colon\mathbb{N}\to \mathbb{N}$ is \emph{almost bijective} if the pre-image of any meager set has asymptotic density $0$. 

By Proposition \ref{prop:CE}, the set of integers which are not $(\epsilon,k)$-normal is a meager set. Thus, if $f\colon\mathbb{N}\to \mathbb{N}$ is almost bijective, then $f(n)$ will be $(\epsilon,k)$-normal for almost all $n$. This gives the following variant on the combinatorial method, which appears explicitly in \cite{PV}:

\begin{thm}
Suppose the function $f\colon\mathbb{N}\to\mathbb{N}$ is almost bijective. 
If, in addition,
\[
m = o\left( \sum_{n=1}^m L(a_n) \right) \qquad \text{ and } \qquad m\cdot \max_{1\le n \le m} L(a_n) =O\left( \sum_{n=1}^m L(a_n)\right) 
\]
then $x=0.\widebar{f(1)}\widebar{f(2)}\widebar{f(3)}\dots$ is normal.
\end{thm}

This result covers a fairly wide variety of functions. De Koninck and K\'{a}tai \cite{DKK11,DKK13} implicitly applied this result with certain variants of the largest prime divisor function. Pollack and Vandehey \cite{PV} showed that one could take $f(n)$ to be various classical number-theoretic functions, including the Euler totient function and the sum-of-divisors function. Sz\"{u}sz and Volkmann \cite{SV} gave fairly general analytic conditions guaranteeing that the values $f(n)$ are $(\epsilon,k)$-normal for almost all $n$. (See the following table for some explicit examples.)

\begin{table}[h]
\begin{centering}

\begin{tabular}{| p{4cm} | p{4cm} | r|}
\hline
Function & Discoverer & Resulting normal number\\
\hline \hline
$P(n)$, the largest prime divisor of $n$ & De Koninck and K\'{a}tai & $0.123253723511213\dots$\\
\hline
$\phi(n)$, the Euler totient function & Pollack and Vandehey & $0.112242646410412\dots$\\
\hline
$\sigma(n)$, the sum of the divisors of $n$ & Pollack and Vandehey & $0.1347612815131812\dots$\\
\hline
$\lfloor n^{1/2} \rfloor$, the floor of the square root of $n$ & Sz\"{u}sz and Volkmann & $0.1112222233333334\dots$\\
\hline
\end{tabular}
\end{centering}
\end{table}

Despite all these results inspired by Besicovitch's work, none of them are strong enough to prove the normality of $x_B=0.14916253649\dots$.\footnote{Szusz and Volkmann mistakenly claim that their result is strong enough to prove a result like this. In Theorem 2 of their paper, they need an additional condition that $\beta \le 1$, because if $\beta>1$ then the bound in line (3.11) would be $M_k=O(1)$, which would cause their condition (v) to fail.}  In particular, for the function $f(n)=n^2$, the pre-image of the meager set $S=\{n^2:n \in \mathbb{N}\}$ is the whole domain $\mathbb{N}$. 

Our goal in the rest of this paper is to update and simplify the proof of the normality of $x_B=0.\widebar{f(1)}\widebar{f(2)}\dots$ with $f(n)=n^2$ and some integer base $g\ge 2$, to reflect modern work on normal numbers, in the hope that it can inspire further combinatorial results. In Section \ref{sec:combproof}, we prove Theorem \ref{thm:combinatorial}. In Sections \ref{sec:half}--\ref{sec:secondhalf}, we present our proof of the normality of $x_B$. In Section \ref{sec:final}, we briefly discuss how our proof differs from Besicovitch's and the difficulties in extending this method combinatorially.

%%%%%%%%%%%%%%%%%
\section{Proof of Theorem \ref{thm:combinatorial}}\label{sec:combproof}
%%%%%%%%%%%%%%%%%%

Consider a sequence $\{a_n\}_{n=1}^\infty$ satisfying the conditions of Theorem \ref{thm:combinatorial}. Let $x=0.\widebar{a_1}\widebar{a_2}\widebar{a_3}\dots$ in the appropriate base $g$. To show that $x$ is normal to base $g$, we must show for any given string $\s$ of length $k$ that 
\[
\lim_{N\to \infty} \frac{\nu(x,N,\s)}{N} = \frac{1}{g^k}.
\]
We fix an $\epsilon>0$, which will be allowed to tend towards zero at the end of the proof.

For a given integer $N$, let $m=m(N)$ be such that the $N$th digit of $x$ lies in the string given by $\widebar{a_m}$. Then
\[
\sum_{n=1}^{m-1} L(a_n) < N \le \sum_{n=1}^m L(a_n).
\]
The second part of \eqref{eq:lengths} implies that $L(a_m) = o(\sum_{n=1}^m L(a_n))$, so we have that $N\sim \sum_{n=1}^m L(a_n)$, $m=o(N)$, and $L(a_m)=o(N)$. Therefore,
\[
\nu(x,N,\s) = \nu(\widebar{a_1}\widebar{a_2}\dots \widebar{a_m},\s) + O(L(f(m)) )=  \nu(\widebar{a_1}\widebar{a_2}\dots \widebar{a_m},\s) + o(N).
\]

The number of times a string of length $k$ can appear in $\widebar{a_1}\widebar{a_2}\dots \widebar{a_m}$ starting in some $\widebar{a_n}$ and ending in some $\widebar{a_{n'}}$ with $n< n'$ is at most $km=o(N)$. Therefore, 
\[
\nu(x,N,\s) = \nu(\widebar{a_1}\widebar{a_2}\dots \widebar{a_m},\s) + o(N) = \sum_{n\le m} \nu(\widebar{a_n},\s) +o(N).
\]

Let $T\subset\mathbb{N}$ be the set of integers $n$ such that $a_n$ is \emph{not} $(\epsilon,k)$-normal. Note that by the assumptions of the theorem, we have $\#\{n\le m: n \in T\} = o(m)$. 
We always have that $\nu(\widebar{a_n},\s)=O(L(a_n))$, and therefore
\begin{align*}
\sum_{\substack{n\le m \\ n \in T}} \nu(\widebar{a_n},\s) &= O\left( \sum_{\substack{n\le m \\ n \in T}} L(a_n) \right) = O\left( \max_{n\le m} L(a_n) \cdot \sum_{\substack{n\le m \\ n \in T}} 1  \right)\\
&= o \left( m \cdot \max_{n\le m} L(a_n) \right) = o\left( N\right).
\end{align*}

Now we let $S=\mathbb{N}\setminus T$ be the set of integers $n$ such that $a_n$ is $(\epsilon,k)$-normal. If $n\in S$ then $\nu(\widebar{a_n},\s) = L(a_n)g^{-k}+O(\epsilon L(a_n))$, and thus
\begin{align*}
\sum_{\substack{n\le m\\ n\in S }} \nu(\widebar{a_n},\s) &= \sum_{\substack{n\le m \\ n\in S}} \left( L(a_n) \left(g^{-k}  +O(\epsilon)\right)\right)\\
&= g^{-k}\left( \sum_{n \le m} L(a_n) - \sum_{\substack{n \le m \\n \in T}} L(a_n) \right) +O\left( \epsilon \cdot  \sum_{\substack{n\le m \\ n\in S}} L(a_n)\right) \\
&= g^{-k} \left( N(1+o(1)) + o(N)\right) +O\left( \epsilon \cdot  \sum_{n\le m } L(a_n)\right)\\
&= g^{-k}N +o(N)+O(\epsilon N)
\end{align*} 

Since the sum over $n\le m$ is equal to the sum over $n\le m$ with $n\in S$ plus the sum over $n\le m$ with $n \in T$, we have shown that
\[
\frac{\nu(x,N,\s)}{N} = g^{-k}+o(1)+O(\epsilon).
\]
Since $\epsilon>0$ was arbitrary, $\nu(x,N,\s)/N\to g^{-k}$ as $N\to \infty$.

%%%%%%%%%%%%%%%%%%%%
\section{Cutting the squares in half}\label{sec:half}
%%%%%%%%%%%%%%%%%%%%

From here on, we shall be interested in the specific case when $a_n=f(n)=n^2$ with a fixed integer base $g\ge 2$. Implicit constants may depend on $g$.

In this case we have $L(f(n))=\lfloor \log_{g} f(n) \rfloor+1=2\log_{g} n +O(1)$, and thus 
\begin{equation}\label{eq:mlengths}
\sum_{n=1}^m L(f(m)) = \frac{2}{\log g} m \log m (1+o(1)).
\end{equation}
Here the sum of $\log_g(n)$ has been estimated using the proof of the integral test. It is clear from \eqref{eq:mlengths} that $a_n=n^2$ satisfies the restrictions on $L(a_n)$ from the statement of Theorem \ref{thm:combinatorial}.

To prove the normality of $x_B=0.\widebar{f(1)}\widebar{f(2)}\widebar{f(3)}\dots$ in this case, it suffices by Theorem \ref{thm:combinatorial} to show that for a fixed $\epsilon>0$ and $k\in \mathbb{N}$, that the number of $n \in [1,m]$ for which $f(n)$ is not $(2\epsilon,k)$-normal is $o(m)$ as $m$ tends to infinity. (Note that using $2\epsilon$ is intentional here.)

Let $\delta=\delta(\epsilon,k)$ be the constant from Proposition \ref{prop:CE}, and let $m'=\lfloor m^{1-\frac{\delta}{2}} \rfloor$. We may ignore values of $n\in [1,m'-1]$, since there are only $o(m)$ of these.

Let $\ell := \lfloor L(f(m))/2 \rfloor$. For sufficiently large $m$, we have $L(f(n))> \ell$ for all $n \in [m',m]$.

We now consider two new auxiliary functions $b(n,m)$ and $c(n,m)$ for a fixed $m$. We let $b(n,m)=\lfloor n^2/g^{\ell} \rfloor$, and we let $c(n,m)$ be the least nonnegative residue of $n^2$ modulo $g^{\ell}$. While we shall define $\widebar{b(n,m)}$ in the usual way as the string of base $g$ digits of $b(n,m)$, we shall modify our definition slightly for $\widebar{c(n,m)}$. If $c(n,m)$ has fewer than $\ell$ base $g$ digits, then append enough $0$'s to the beginning of the string $\widebar{c(n,m)}$ so that it has length $\ell$.

With these definitions, the string $\widebar{f(n)}$ is the concatenation of $\widebar{b(n,m)}$ and $\widebar{c(n,m)}$, for all  $n\in [m',m]$. As a quick example, consider $m=500$ and $n=179$. Then $f(m)=250000$, so that $l=3$, and $f(n)=32041$. In this case, $b(n,m)=32$ and $c(n,m)=41$, so that $\widebar{b(n,m)}=32$ and $\widebar{c(n,m)} =041$. % Therefore, we do truly have that $\widebar{f(n)}$ is the concatenation of $\widebar{b(n,m)}$ and $\widebar{c(n,m)}$ in this case.

Since $\widebar{f(n)}$ has close to $2\ell$ digits, and both $\widebar{b(n,m)}$ and $\widebar{c(n,m)}$ have approximately $\ell$ digits, we may think of this as bisecting $f(n)$ into halves.

Now we make two claims which we will prove in subsequent sections:
\begin{claim}
The number of $n\in [m',m]$ for which $\widebar{b(n,m)}$ is not $(\epsilon,k)$-normal is $o(m)$ as $m$ tends to infinity. 
\end{claim}
\begin{claim}
The number of $n \in [m',m]$ for which $\widebar{c(n,m)}$ is not $(\epsilon,k)$-normal is $o(m)$ as $m$ tends to infinity.
\end{claim}

We now finish proving the normality of $x_B$ assuming these two claims.

Suppose that both $\widebar{b(n,m)}$ and $\widebar{c(n,m)}$ are $(\epsilon,k)$-normal. Then 
\begin{align*}
\nu(f(n),\s) &= \nu(\widebar{b(n,m)},\s)+\nu(\widebar{c(n,m)},\s)+O(k)\\
&= L(\widebar{b(n,m)})(g^{-k} + O(\epsilon )) + L(\widebar{c(n,m)})(g^{-k}+O(\epsilon ))+O(k)\\
&= L(f(n))(g^{-k}+O(\epsilon))+O(k).
\end{align*}
(Here we are using the big-O notation with implicit constant $1$ in all cases.) Since $k = O(\epsilon L(f(n)))$ for all $n\in [m',m]$ provided $m$ is sufficiently large, we have that $f(n)$ is $(2\epsilon,k)$-normal in this case. So $f(n)$ is not $(2\epsilon,k)$-normal only if $\widebar{b(n,m)}$ or $\widebar{c(n,m)}$ is not $(\epsilon,k)$-normal, and there are only $o(m)$ such $n$ in the interval $[m',m]$ by our two claims. This completes the proof.

%%%%%%%%%%%%%%%%%%%%%%%%
\section{How often is the first half of $n^2$ normal?}\label{sec:firsthalf}
%%%%%%%%%%%%%%%%%%%%%%%%

Here we will prove that the number of $n\in [m',m]$ for which $\widebar{b(n,m)}$ is not $(\epsilon,k)$-normal is $o(m)$ as $m$ tends to infinity. This is comparatively simple. As the first half of the digits of $f(n)$ grow fairly regularly, they cannot take any given non-$(\epsilon,k)$-normal value too frequently.

Recall that $b(n,m)=\lfloor n^2/g^{\ell} \rfloor$. For the remainder of this section we will often suppress the dependence on $m$ and just write $b(n)$.

Since $\ell \ge \log_{g} m - 1$, we have $g^\ell \ge m/g$, and thus $b(n)$ is always in the interval $[1,gm]$. By the first half of Proposition \ref{prop:CE}, there are at most $(gm)^{1-\delta}$ integers in the interval $[1,gm]$ that are not $(\epsilon,k)$-normal.

 Now suppose that $m'\le n_1< n_2 \le m$. Then, since we have $\ell \le \log_g m+1$ as well, 
\begin{align*}
b(n_2)-b(n_1) &= \frac{n_2^2-n_1^2}{g^{\ell}} +O(1) = g^{-\ell} (n_2-n_1)(n_2+n_1) +O(1)\\
& \ge (gm)^{-1} (n_2-n_1)(2m') +O(1) 
\end{align*}
Thus, if $n_2-n _1 \ge m^{3\delta/4}$, then
\[
b(n_2)-b(n_1)\ge (gm)^{-1} \cdot m^{3\delta/4} \cdot 2m' +O(1) \ge 2g^{-1}m^{\delta/4}+O(1).
\]
Thus, for sufficiently large $m$, we have $b(n_1)\neq b(n_2)$ for $n_2-n_1 \ge m^{3\delta/4}$. Since $b(n)$ is non-decreasing, this means that $b(n)$ can take a given value in the interval $[1,gm]$ at most $m^{3\delta/4}$ times. Since there are at most $(gm)^{1-\delta}$ integers in the interval $[1,gm]$ that are not $(\epsilon,k)$-normal, we have that at most $O(m^{1-\delta/4})=o(m)$ of the integers $n\in [m',m]$ have $\widebar{b(n,m)}$ not $(\epsilon,k)$-normal.

%%%%%%%%%%%%%%%%%%%%%%%%%%%%%
\section{How often is the second half of $n^2$ normal?}\label{sec:secondhalf}
%%%%%%%%%%%%%%%%%%%%%%%%%%%%

Here we will prove that the number of $n\in [m',m]$ for which $\widebar{c(n,m)}$ is not $(\epsilon,k)$-normal is $o(m)$ as $m$ tends to infinity. As before, we will write $c(n)$ in place of $c(n,m)$.

Let $B$ be the set of integers $b\in [0,g^{\ell}-1]$ such that $\widebar{b}$ is not $(\epsilon,k)$-normal. Here again, we assume $\widebar{b}$ to be padded with initial zeros so as to have length $\ell$. By the second half of Proposition \ref{prop:CE}, the cardinality of $B$ is at most $g^{\ell(1-\delta')}$ for a certain $\delta'=\delta'(\epsilon,k)>0$.

How often is $c(n)\in B$? Since $[m',m] \subset [1,g^{\ell+1}]$, the number of $n\in [m',m]$ with $c(n) \in B$ is at most the $g$ times the count of such $n$ in $[1,g^{\ell}]$. By Cauchy--Schwarz, 
\begin{align*}
\sum_{\substack{1 \le n \le g^{\ell}\\c(n)=b}} 1 &= \sum_{b\in B} \#\{1\le n \le g^{\ell}\mid n^2 \equiv b \pmod{g^{\ell}} \} \\
&\le\left( \sum_{b\in B} 1\right)^{1/2} \left( \sum_{b\in B} \#\{1\le n_1,n_2 \le g^{\ell}\mid n_1^2 \equiv n_2^2 \equiv b \pmod{g^{\ell}} \}\right)^{1/2}.
\end{align*} 

We are now left with the problem of counting how many pairs of integers $(x,y)$ there are with $1\le x,y\le g^l$ and $x^2\equiv y^2 \pmod{g^\ell}$. When $g$ is a prime power, a satisfactory answer is contained in the next lemma.

\begin{lem}\label{lem:secondhalf}
Let $p$ be a prime and $e$ be a positive integer. The number of solutions to the congruence  $x^2 \equiv y^2 \pmod{p^e}$ is at most
\[
\begin{cases}
2e \cdot p^e, & \text{if }p\text{ is odd,}\\
4e \cdot p^e, & \text{if }p = 2.
\end{cases}
\]
\end{lem}

\begin{proof} Certainly, if $p^{\lceil e/2\rceil}$ divides each component of the pair $(x,y)$, then $x^2\equiv y^2\pmod{p^e}$. There are $(p^{e-\lceil e/2\rceil})^2$ solutions of this kind. 

For all other solutions, $p^{\lceil e/2\rceil}$ divides neither component. Group the remaining solutions according to the largest exponent $r$ for which $p^r \mid x$. Then $0 \leq r < e/2$. Since $p^{2r}\mid p^e \mid x^2-y^2$ and $p^{2r}$ divides $x^2$, we see that $p^r$ divides $y$. Write $x=p^r x'$ and $y=p^r y'$, and notice that determining the pair $(x,y)$ modulo $p^e$ amounts to determining the pair $(x', y')$ modulo $p^{e-r}$. Now $x^2 \equiv y^2\pmod{p^e}$ precisely when \begin{equation}\label{eq:congsq} x'^2 \equiv y'^2\pmod{p^{e-2r}}.\end{equation}
This final congruence looks very similar to the one we started with, and one might wonder if we have gained anything. Indeed we have: by the maximality of $r$, we know that $x'$ must be coprime to $p$, which forces $y'$ to be coprime to $p$ as well. In other words, $x'$ and $y'$ represent elements of the unit group modulo $p^{e-2r}$. 

The key word here is ``group." As was known to Gauss, if $p$ is an odd prime, the units group modulo $p^{e-2r}$ is cyclic of order $p^{e-2r}(1-1/p)$. In any cyclic group of even order, each element has precisely two square roots. So if $p$ is odd, the congruence \eqref{eq:congsq} has  $2p^{e-2r}(1-1/p)$ solutions modulo $p^{e-2r}$, and so has $p^{2r} \cdot 2p^{e-2r}(1-1/p) = 2p^e (1-1/p)$ solutions modulo $p^{e-r}$. Putting everything together, we see that the number of solutions to $x^2\equiv y^2\pmod{p^e}$ is exactly
\[ 2p^e (1-1/p) \lceil e/2\rceil + (p^{e-\lceil e/2\rceil})^2. \]
Here we used that the number of integers in the range $0 \leq r < e/2$ is precisely $\lceil e/2\rceil$.

What if $p=2$? In this case, the group of units modulo $p^{e-2r}$ is either cyclic or the direct sum of two cyclic groups, and so each element can have at most four square roots. Modifying the above argument accordingly, we find that the number of solutions in this case is at most $4p^e (1-1/p) \lceil e/2\rceil + (p^{e-\lceil e/2\rceil})^2$. 

Finally, it is straightforward to check that these bounds do not exceed the upper bounds specified in the statement of the lemma.
\end{proof}

Now suppose that $g$ has the prime factorization $p_1^{e_1}p_2^{e_2}\dots p_j^{e_j}$. % where $j$ is the number of distinct prime factors of $g$. 
We use the above lemma combined with the Chinese Remainder Theorem to see that the number solutions to $x^2\equiv y^2 \pmod{g^\ell}$ is at most $2(\prod_{i=1}^{j} 2e_i \ell) g^{\ell}$, where the initial $2$ comes from the possibility that $2$ is a factor of $g$.
Since $g^\ell \le gm$ and $\ell \le \log_g m+1$, the total number of pairs is $O((\log m)^j m)$.

Thus, the number of $n \in [m',m]$ for which $c(n)$ belongs to $B$ is is $O(\sqrt{\#B} \cdot (\log{m})^{j/2} m^{1/2})$. Recalling that $\#B \le g^{\ell(1-\delta')}$ while $g^{\ell} \leq gm$, we see that $\#B = O(m^{1-\delta'})$. Since $(\log{m})^{j/2}$ is smaller than $m^{\delta'/4}$ for large $m$, our final count of $n$ is $O(m^{1-\delta'/4})$, which is $o(m)$. This completes the proof.

%%%%%%%%%%%%%%%%%%%%%%
\section{Revisiting and extending Besicovitch}\label{sec:final}
%%%%%%%%%%%%%%%%%%%%%%

As mentioned earlier, the main result of Besicovitch's paper was not a proof that $x_B=0.1491625\cdots$ is normal base $10$. What Besicovitch showed was that, almost all of the integers $n^2$, for $n\in \mathbb{N}$, are $(\epsilon,1)$-normal to a given base $g\ge 2$. 

Besicovitch, like we did above, split the string $\widebar{n^2}$ into approximate halves. His method for showing that the first half has good normality properties is very similar to the method we used. His method for the second half is a long direct counting argument quite different from ours. Besicovitch relies heavily on results from Diophantine approximation about how well real numbers can be approximated by rational numbers with small denominators.

It's natural to ask if we could show combinatorially that  $n^3$ is $(\epsilon,k)$-normal for almost all $n$. We could divide the string $\widebar{n^3}$ into $3$ roughly equal length pieces as we did with $\widebar{n^2}$. The methods of Section \ref{sec:firsthalf} (and those of Besicovitch) would show that the first third of $\widebar{n^3}$ would almost always be $(\epsilon,k)$-normal. The methods of Section \ref{sec:secondhalf}, with minor modifications\footnote{We could count the $n$ for which $p_i^{\lceil \log \ell \rceil} \mid c(n)$ for some $i$ first, and this will be $o(m)$. Then we apply our Cauchy-Schwarz estimate to the remaining $n$, so that in Lemma \ref{lem:secondhalf} we could assume that $p^{\lceil \log \ell \rceil}$ does not divide either $x$ and $y$. The rest of the proof would be mostly unchanged.}, would show the same for the last third of $\widebar{n^3}$. (Besicovitch's methods would not work here due to the limits of Diophantine approximation of real numbers.) However, neither our methods nor Besicovitch's would be sufficient to show the middle third of $\widebar{n^3}$ is almost always $(\epsilon,k)$-normal. New techniques are required for that.

\end{document}